\documentclass[12pt, A4, babel]{article}
\usepackage{amssymb, amsfonts, amsmath}
\usepackage{amsthm}
\usepackage[mathcal]{eucal}
\usepackage{mathrsfs}
\usepackage[T2A]{fontenc}

\usepackage{amssymb, amsmath}
\usepackage{amsthm}
\newtheorem{thm}{Theorem}

\newtheorem{lem}{Lemma}

\newtheorem{defin}{Definition}

\title{A continuum of non-isomorphic 3-generator groups with probabilistic law $x^n=1$}
\author{V.S. Atabekyan \thanks{avarujan@ysu.am} \and A.A. Bayramyan \thanks{arman.bayramyan@ysu.am} \and V.H. Mikaelian \thanks{vmikaelian@ysu.am}}

\date{}

\usepackage[dvipsnames]{xcolor}

\begin{document}
	\maketitle

	\begin{abstract}
		
		In this paper, we construct a continuum family of non-isomorphic 3-generator groups in which the identity $x^n = 1$ holds with probability 1, while failing to hold universally in each group. This resolves a recent question about the relationship between probabilistic and universal satisfaction of group identities. Our construction uses $n$-periodic products of cyclic groups of order $n$ and two-generator relatively free groups satisfying identities of the form $[x^{pn}, y^{pn}]^n = 1$. We prove that in each of these products, the probability of satisfying $x^n = 1$ is equal to 1, despite the fact that the identity does not hold throughout any of these groups.
		
	\end{abstract}
	
	\section*{Introduction} 
	
	The interplay between group theory and probability theory has become a vibrant area of study in modern algebra. One application of probability theory in group theory is the study of probabilistic statements that provide alternative characterizations of groups. One of the earliest such results is the following observation: if the probability that two randomly chosen elements of a finite (or compact) group commute is greater than $\frac{5}{8}$, then the group is abelian (see \cite{Gus}). 
	A related result states that if 
	$G$ is a finite group in which the probability that two randomly chosen elements generate a solvable subgroup exceeds $\frac{11}{30}$, then $G$ itself is solvable \cite{GW}.
	
	These kinds of results can be extended by introducing probability measures on finitely generated groups. A natural approach is to fix a generating set, and consider a sequence of uniform measures on the $k$-balls of the Cayley graph of the group, then take the limit as $k\rightarrow\infty$. In the following we provide the general definition for finitely generated groups following \cite{ABGK}.
	
	Let $M = \{\mu_k\}$ be a sequence of probability measures on a finitely generated group $G$, and let $w$ be an arbitrary word in the free group $F_m$ of rank $m$. We define the probability that $G$ satisfies the group identity (law) $w = 1$ with respect to the sequence of measures $M$ as the number
	\begin{equation}\label{d}
		\mathbb{P}_{M}(w = 1 \; \textrm{in} \; G) = \limsup_{k \to \infty} {\mu_k} (\{(g_1,\dots,g_m)\in G^m\;:\;w(g_1,\dots,g_m)=1\}).
	\end{equation}
	
	If
	\[
	{\mathbb{P}}_{M}(w = 1 \; \textrm{in} \; G) = 1,
	\]
	then we say that the word $w$ is an $M$-probabilistic group identity (almost identity) for the group $G$. The value of the probability ${\mathbb{P}}_{M}(w = 1 \; \textrm{in} \; G)$ conveys certain information about the group $G$. For example, as shown in \cite{T} and \cite{ABGK}, respectively, if $\mathbb{P}([x, y] = 1 \; \textrm{in} \; G) > 0$ or $\mathbb{P}(x^2 = 1 \; \textrm{in} \; G) > 0$, then the group is virtually abelian (see also \cite{YMV}). A detailed survey of results on this topic can be found in \cite{ABGK}.
	
	In \cite{ABGK} the following question is formulated (see Question 13.3): Let $w \in F_m$ be an arbitrary word, $G$ be a finitely generated group with generating set $S$, and $R_n^{(1)}, R_n^{(2)}, \ldots, R_n^{(d)}$ be independent random walks on  $G$ with respect to $S$. If
	\[
	\lim_{n \to \infty} \mathbb{P}\big(w(R_n^{(1)}, R_n^{(2)}, \ldots, R_n^{(d)}) = 1 \; \text{in} \; G\big) = 1,
	\]
	does it follow that $w = 1$ is an identity in $G$?
	
	According to the well-known work of S.I. Adian \cite{A82}, for all $m\geq 2$ and odd $n\geq 665$, the symmetric random walk on the free Burnside groups $B(m,n)$ is transient. Based on this result, a potential counterexample to the stated problem can be sought among the class of groups defined by relations of the form $A^n$.
	
	This question has been considered in \cite{O+} and \cite{AB}, where the authors, using different approaches, construct finitely generated counterexamples to the stated problem. In particular, in \cite{AB} a 3-generator group is constructed in which $x^n = 1$ is a probabilistic identity for all sufficiently large odd $n$, but which contains an infinite cyclic subgroup. This means that the identity $x^n = 1$ does not hold throughout the entire group.
	
	In this work, we will strengthen this result by constructing a continuum family of non-isomorphic 3-generator groups, each of which has the property that the probability of the identity $x^n = 1$ is equal to 1. At the same time, this identity does not hold in any of these groups. As the probability measures on $G$ we will fix the natural sequence of probability measures on the Cayley graph of the group $G$. Let $B_{G,S}(k)$ denote the ball of radius $k$ in the Cayley graph centered at the identity element of the group $G$ with respect to some fixed symmetric generating set $S$, and let  $M = {\mu_k}$, where $\mu_k$ is the uniform distribution on $B_{G,S}(k)$, meaning $\mu_k(g) = \frac{1}{|B_{G,S}(k)|}$ if $g \in B_{G,S}(k)$, and $\mu_k(g) = 0$ otherwise.
	
	Our approach is based on two constructions proposed by S.I. Adian: the concept of an $n$-periodic product developed in \cite{A76}, and the infinite system of independent group identities constructed in \cite{A70i} and \cite[Ch. VII]{A}. More specifically, we consider a system of identities of the form $\{[x^{pn}, y^{pn}]^n = 1\}$, where $n \geq 1003$ is a fixed odd number, and $p$ runs over a fixed set $\mathcal{P}$ of prime numbers. Let $\mathbb{F}(n, \mathcal{P})$ denote the free group of rank 2 with free generators $b_1, b_2$ in the variety of groups satisfying all identities $\{[x^{pn}, y^{pn}]^n = 1\}$ for $p \in \mathcal{P}$. According to S.I. Adian's result \cite{A} for different sets of prime numbers $\mathcal{P}$ the groups $\mathbb{F}(n, \mathcal{P})$ are non-isomorphic. Next, we consider the $n$-periodic product $\mathbb{A}_\mathcal{P} = \mathbb{F}(n, \mathcal{P}) \ast^\mathbf{n} Z_n$, where $Z_n$ is a cyclic group of order $n$ with generator $a$. Clearly, $\mathbb{A}_\mathcal{P}$ is a 3-generator group.
	
	Now let us formulate our main result:
	
	\begin{thm}\label{T} For any sufficiently large fixed odd number $n$, the following holds:
		\begin{enumerate}
			\item In the group $\mathbb{A}_\mathcal{P} = \mathbb{F}(n, \mathcal{P}) \ast^\mathbf{n} Z_n$ the identity $x^n = 1$ is an $M$-probabilistic identity.
			\item In the group $\mathbb{A}_\mathcal{P}$ the identity $x^n = 1$ does not hold.
			\item There exists a continuum family of non-isomorphic 3-generator groups of the form $\mathbb{A}_\mathcal{P} = \mathbb{F}(n, \mathcal{P}) \ast^\mathbf{n} Z_n$, corresponding to different subsets $\mathcal{P}$ of prime numbers.
		\end{enumerate}
		
	\end{thm}
	The groups $\mathbb{A}_\mathcal{P}$ in Theorem \ref{T} provide a negative answer to the question mentioned above in a significantly strengthened form.
	
	The paper is organized as follows. In Section \ref{IIT} we outline the key properties of the groups $\mathbb{F}(n, \mathcal{P})$, which form the foundation of our construction. Section \ref{PP} reviews some relevant facts about $n$-periodic products. In Section \ref{n-T} we demonstrate the necessary properties of the $n$-torsion groups. Section \ref{4} provides a growth estimate for 3-generator free Burnside groups, a crucial ingredient in the proof of the main theorem. In Section \ref{5} we combine these tools to prove the main theorem.
	
	The current results are related to \cite{M07, M10, M18} in which we found a continuum of 3-generator soluble non-Hopfian (non-metanilpotent) groups that generate distinct varieties of groups. This is a wide class of groups for which Higman's Problem 16 in \cite{N67} has a positive answer.
	
	Throughout the paper $n$ is assumed to be a large enough odd number. For each result, we specify the range of $n$ for which it is valid.

	\section{Infinite independent system of group \\identities with two variables}\label{IIT}
	
	The finite basis problem, posed by B. Neumann in 1937, asks whether there exist infinite irreducible systems of group identities, i.e., systems in which no relation is a consequence of the others. In 1969, S.I. Adian constructed examples of such systems with two variables, resolving this problem (\cite{A70i}). This result was later included in his 1975 monograph \cite{A}, where he demonstrated that for any odd $n \geq 1003$, the following family of two-variable identities is irreducible:
	\begin{equation}
		\label{id}
		\{[x^{pn}, y^{pn}]^n = 1\},
	\end{equation}
	where the parameter $p$ runs over all prime numbers (see \cite[Chapter VII]{A}, Theorem 2.1). From this result, it follows that for any odd $n \geq 1003$ there exists a continuum of distinct varieties $\mathcal{A}_n(\mathcal{P})$ corresponding to different subsets $\mathcal{P}$ of the set of prime numbers. Consequently, the relatively free groups $ \mathbb{F}(n, \mathcal{P})$ of rank 2 in the varieties $\mathcal{A}_n(\mathcal{P})$ for different sets $\mathcal{P}$ are non-isomorphic. Additional properties of these groups have been studied in \cite{AA17}. As shown in \cite[Section 2, Chapter VII]{A} (see also \cite{AA17}), the group has the following presentation:
	\begin{equation}\label{G}
		\mathbb{F}(n, \mathcal{P})=\left\langle b_1, b_2\,|\, A^n=1, A\in \mathcal{E}=\bigcup^\infty_{\alpha=1}\mathcal{E}_\alpha\right\rangle,
	\end{equation}
	where the words $A \in \mathcal{E}_\alpha$ are chosen in a specific way and are called marked elementary periods (for a detailed definition of elementary periods, the sets $\mathcal{K}_{\alpha}, \overline{\mathcal{M}}_{\alpha}$, etc., see \cite{A70i}, \cite{A}, \cite{AA17}).
	
	For simplicity we will denote $\mathbb{F}(n, \mathcal{P})$ by  $\mathbb{F}$ for a fixed pair $(n, \mathcal{P})$.
	
	The following three lemmas were established in \cite{AA17} $( n \geq 1003 )$:
	
	\begin{lem}[Lemma 8, \cite{AA17}]\label{ls}
		For any word $X$ that is not equal to 1 in the group $\mathbb{F}$, there exist words $T$ and $A$ such that $X = TA^rT^{-1}$ in $\mathbb{F}$ for some integer $r$, where $A \in \mathcal{E}$ or $A$ is an unmarked elementary period of some rank, and $A^q$ appears in some word from the class  $\overline{\mathcal{M}}_{\gamma-1}$.
	\end{lem}
	
	\begin{lem}[Lemma 4, \cite{AA17}]\label{lo}
		If $A$ is a marked elementary period of some rank $\gamma$, and $A^q$ appears in some word from the class $\mathcal{K}_{\gamma-1}$, then $A$ has order $n$ in the group $\mathbb{F}$.
	\end{lem}
	
	\begin{lem}[Lemma 5, \cite{AA17}]\label{ln}
		If $A$ is an unmarked elementary period of some rank $\gamma$, and $A^q$ appears in some word from the class $\overline{\mathcal{M}}_{\gamma-1}$, then $A$ has infinite order in $\mathbb{F}$.
	\end{lem}
	
	From Lemmas \ref{ls}, \ref{lo}, and \ref{ln} it directly follows:
	
	\begin{lem}\label{f}
		Every element of the group $\mathbb{F} $ is conjugate either to a power of an element of infinite order or to a power of an element of order $n$.
	\end{lem}

	\section{$n$-periodic products}\label{PP}
	
	The concept of $n$-periodic products of groups was introduced by Adian in \cite{A76}. He demonstrated that the operation of $n$-periodic products with an odd exponent $n \geq 665$ possesses several important properties: it is exact, associative, and satisfies the hereditary property for subgroups. These properties are also characteristic of direct and free products, providing a solution to the well-known Maltsev problem (see also \cite{AA17i} for details). This operation, denoted by $\prod_{i \in I}^{\mathbf{n}} G_i$, is defined for odd $n \geq 665$ as the quotient group of the free product of a given family of groups $\{G_i\}$ by a specific normal subgroup determined by a system of relations of the form $A^n = 1$.
	
	The following two lemmas from \cite{A76} reveal key properties of $n$-periodic products for odd $n \geq 665$.
	
	\begin{lem}[Theorem 1, \cite{A76}]\label{l1}
		The groups $G_i, i \in I$ canonically embed into $\prod_{i \in I}^{\mathbf{n}} G_i $ as subgroups.
	\end{lem}
	
	\begin{lem}[Theorem 7, \cite{A76}]\label{spp}
		For every element $x \in \prod_{i \in I}^{\mathbf{n}} G_i$, either $x^n = 1$ in $\prod_{i \in I}^{\mathbf{n}} G_i$, or $x$ is conjugate to an element of some subgroup $G_i$ of the group $\prod_{i \in I}^{\mathbf{n}} G_i$.
	\end{lem}
	
	As shown in \cite{AA15}, the statement of Lemma \ref{spp} is characteristic for $n$-periodic products, i.e. the $n$-periodic product of a given family of groups is uniquely determined by the  properties indicated in Lemma \ref{spp}. 
	
	\section{$n$-torsion groups}\label{n-T}
	
	Let $S$ be a group alphabet, $\mathcal{R}$ be a set of words over this alphabet, and let $n > 1$ be a fixed natural number. Consider a group $G$ defined by the presentation
	\begin{equation}\label{e1}
		G = \langle S \;|\; R^n = 1,\, R \in \mathcal{R} \rangle.
	\end{equation}
	
	\begin{defin}\label{n}
		We say that the group \eqref{e1} is an $n$-torsion group (see \cite{AA19}, Definition 1.1) or a partial Burnside group of exponent $n$ (see \cite{Bo}, Definition I.1), if for any element $y \in G $ either $y^n = 1$ or $y$ has infinite order.
	\end{defin}
	
	The simplest examples of $n$-torsion groups are the cyclic group of order $n$ and the infinite cyclic group. It is also clear that free Burnside groups and absolutely free groups of any rank are $n$-torsion groups for any natural $n$.
	
	From the equality \eqref{G} it follows that each of the groups $\mathbb{F}(n, \mathcal{P})$ has a presentation of the form \eqref{e1}. Therefore, by Definition \ref{n} and Lemma \ref{f} we conclude:
	
	\begin{lem}\label{1.0}
		The groups $\mathbb{F}(n, \mathcal{P})$ are $n$-torsion groups.
	\end{lem}
	
	The following statement was proven in \cite{AA19}.
	
	\begin{lem}[Theorem 1.1, \cite{AA19}]\label{1.1}
		The $n$-periodic product of any family of $n$-torsion groups is itself an $n$-torsion group $(n \geq 665)$.
	\end{lem}
	
	Combining Definition \ref{n} with lemmas \ref{1.0} and \ref{1.1}, we immediately obtain:
	
	\begin{lem}\label{An}
		The groups $\mathbb{A}_\mathcal{P} = \mathbb{F}(n, \mathcal{P}) \ast^\textbf{n} Z_n$ are $n$-torsion groups $( n \geq 1003)$.
	\end{lem}
	
	Let $G$ be an arbitrary $n$-torsion group (i.e., a partial Burnside group of exponent $n$), where $n$ is a large enough odd number (for example $n>10^{100}$). According to Proposition III.4 in \cite{Bo}, any partial Burnside group $G$ of exponent $n$ can be presented axiomatically in the form $G = B_C(S, n) = \langle S \;|\; R^n = 1, \, R \in C \rangle$, where $C$ is some partial Burnside set in the alphabet $S$ (see Definitions II.45 and II.46, \cite{Bo}). On the other hand, according to Theorem III.3 in \cite{Bo} $G$ also admits a special graded presentation of the form
	\begin{equation}\label{gc}
		G = G_C(\infty) = \langle S \;|\; A^n = 1, \, A \in R \rangle,
	\end{equation}
	which is known as a minimal partial Burnside presentation (MPBP) of $G$ (see Definition II.47, \cite{Bo}). 
	
	For the group $\mathbb{A}_\mathcal{P}$ let us fix the generating set $S = \{a^{\pm1}, b_1^{\pm1}, b_2^{\pm1}\}$, and consider the corresponding MPBP \eqref{gc}. Since all statements in Chapters II--IV of \cite{Bo} hold for partial Burnside groups, they also apply to $\mathbb{A}_\mathcal{P}$. Denote by $S'=\{b_1^{\pm1}, b_2^{\pm1}\}$ the generating set of the group $\mathbb{F}$.
	
	We will need Lemmas \ref{w}--\ref{28} proven in \cite{Bo} on the basis of the monograph \cite{O91} for arbitrary partial Burnside groups. We state it here specifically for the group $\mathbb{A}_\mathcal{P}$.
	
	\begin{lem}[Lemma IV.3, \cite{Bo}]\label{w}
		Let  $W$ be an arbitrary word over the alphabet $S$. Suppose $U$ is the shortest word such that $W$ is conjugate to some power of $U$ in the group $\mathbb{A}_\mathcal{P}$. If a letter $s \in S$ occurs in $U$, then it must also occur in $W$.
	\end{lem}
	
	\begin{lem}[Lemma II.49 \cite{Bo} (see also Lemma 18.1, \cite{O91}\label{49} )]Every word is conjugate in rank $i\geq0$ to a power of some period of rank $j\le i$ or to a power of a simple word in rank $i$.
	\end{lem}
	
	\begin{lem}[Proposition III.9 \cite{Bo}]\label{III} If a word $X$ representing a nontrivial element of $\mathbb{A}_\mathcal{P}$ is conjugate in $\mathbb{A}_\mathcal{P}$ to a power of a period $U$, then $|U|\le|X|$.
	\end{lem}
	
	\begin{lem}[Theorem II.57 \cite{Bo} (see also Lemma 19.5, \cite{O91})\label{57}] Let $p$ be a section of the contour of a reduced diagram $D$ whose label
		is $A$-periodic, where $A$ is simple in rank $r(D)$ or is a period of rank $k\le r(D)$, and in the latter case $D$ has no cells of rank $k$ $A$-compatible with $p$. (If $p$ is a cyclic section, then we further require that $Lab(p)=A^m$ for some integer $m$). Then $p$ is a smooth section of rank $|A|$ in the contour of $D$.
	\end{lem}
	
	\begin{lem}[Theorem II.26 \cite{Bo} (see also Theorem 17.1\cite{O91})]\label{26} Let $D$ be a circular $A$-map with contour $qt$ or an annular $A$-map
		with contours $q$ and $t$. If $q$ is a smooth section, then $\overline{\beta}|q|\le|t|$ (equality holds if and only if $|q| = |t| = 0$).
	\end{lem}
	
	\begin{lem}[Lemma II.28 \cite{Bo} (see also Lemma 17.1 \cite{O91})]\label{28} Let $D$ be an annular $A$-map with contours $p$ and $q$. Then there is a path $t$ connecting vertices $o_1$
		and $o_2$ of the paths $p$ and $q$, respectively, such that $|t| < \gamma(|p|+|q|)$.
	\end{lem}
	
	The following statement follows directly from Lemma \ref{w}.
	\begin{lem}\label{geo2}
		Let \( V \) be a word over the alphabet \( S \) in the group $\mathbb{A}_\mathcal{P}$, and suppose that \( V \) is conjugate to some word \( W \) over the alphabet \( S' \). Assume that \( U \) is a shortest word such that \( V \) is conjugate to a power of \( U \) in $\mathbb{A}_\mathcal{P}$. Then the word \( U \) is a word over the alphabet \( S' \).
	\end{lem}
	
	\begin{proof}
		If \( U \) is the shortest word such that \( V \) is conjugate to a power of \( U \) in $\mathbb{A}_\mathcal{P}$, then \( U \) is also the shortest word such that \( W \) is conjugate to a power of \( U \) in $\mathbb{A}_\mathcal{P}$. By Lemma~\ref{w}, any letter \( s \in S \) occurring in \( U \) must also occur in \( W \). Since \( W \) contains only letters from the set \(S'= \{b_1^{\pm1}, b_2^{\pm1}\} \), it follows that \( U \) is also a word over the alphabet \( \{b_1^{\pm1}, b_2^{\pm1}\} \).
	\end{proof}
	
	\begin{lem}\label{geo3}
		Let \( V \) be a word over the alphabet \( S \), and suppose that \( V \) is conjugate to some word \( W \) over the alphabet \( S' \) in the group $\mathbb{A}_\mathcal{P}$. Then there exists a cyclic shift \( V_1 \) of \( V \) such that
		\[
		V_1 = T U^k T^{-1},
		\]
		where \( U \) is a word over the alphabet \( S' \), and the following inequalities hold:
		\[
		\overline{\beta}|U^k| < |V| \quad \text{and} \quad |T| < \gamma(|V| + |U^k|).
		\]
	\end{lem}
	
	\begin{proof}
		Let \( U_1 \) be the shortest word such that \( V \) is conjugate to a power of \( U_1 \) in $\mathbb{A}_\mathcal{P}$. By Lemma~\ref{geo2}, \( U_1 \) is a word over the alphabet \( \{b_1^{\pm1}, b_2^{\pm1}\} \). By Lemma~\ref{49}, the word \( U_1 \) is conjugate to \( U^t \) for some word \( U \), where either \( U \) is an elementary period of some rank or a word that is simple in every rank. In the first case, Lemma~\ref{III} implies \( |U| \le |U_1| \); in the second case, the inequality \( |U| \le |U_1| \) follows from the definition of a simple word. Since \( U_1 \) is the shortest word such that \( V \) is conjugate to a power of \( U_1 \), we must have \( |U| = |U_1| \). Therefore, by Lemma~\ref{geo2}, \( U \) is also a word over the alphabet \( \{b_1^{\pm1}, b_2^{\pm1}\} \).
		
		Consider a minimal annular diagram \( \Delta \) with boundary paths \( p \) and \( q \) labeled by \( U^k \) and \( V^{-1} \), respectively. By Lemma~\ref{57}, the path \( p \) is smooth. Thus, Lemma~\ref{26} yields:
		\[
		\overline{\beta}|p| = \overline{\beta}|U^k| < |V|.
		\]
		By Lemma~\ref{28}, the diagram \( \Delta \) can be cut along a path \( t \) from a point \( o_1 \) on \( p \) to a point \( o_2 \) on \( q \), producing a disc diagram with boundary label
		\[
		V_1 = T U_2^k T^{-1},
		\]
		where \( V_1 \) is a cyclic shift of \( V \), \( U_2 \) is a cyclic shift of \( U \), and \( T \) is the label of the path \( t \). Moreover,
		\[
		|T| = |t| < \gamma(|p| + |q|) = \gamma(|V| + |U^k|).
		\]
		Without loss of generality, we may assume that \( U_2 = U \), since any cyclic shift of a word over \( S' \) is still a word over \( S' \).
	\end{proof}
	
	Let \( B_{\mathbb{A}_\mathcal{P}}(r) \) denote the ball of radius \( r \) in the Cayley graph of $\mathbb{A}_\mathcal{P}$ with respect to the generating set \(S \), and let \( B_\mathbb{F}(r) \) denote the ball of radius \( r \) in the Cayley graph of \( \mathbb{F} \) with respect to its generating set \( S' \).
	
	From Lemma \ref{geo3} it follows
	\begin{lem}\label{geo4}
		Let \( V \) be a word over the alphabet \( S^{\pm1} \) in the group $\mathbb{A}_\mathcal{P}$, and suppose that \( V \) is conjugate to a word \( W \) over the alphabet \( S' \). Then
		\[
		V = X T U^k T^{-1} X^{-1},
		\]
		where:
		\begin{itemize}
			\item \( U^k \) represents an element \( u \in B_\mathbb{F}\left( \tfrac{1}{\overline{\beta}} |V| \right) \),
			\item \( T \) represents an element \( z \in B_{\mathbb{A}_\mathcal{P}}(\gamma(1+\tfrac{1}{\overline{\beta}})|V|) \),
			\item \( X \) is a prefix or suffix of the word \( V \).
		\end{itemize}
	\end{lem}
	
	We use the least parameter principle (LPP) (see pages 165--166 \cite{O91} or pages 8--9 \cite{Bo}). Taking into account that a word of length \( r \) has exactly \( r \) cyclic shifts, Lemma~\ref{geo4} implies the following result:
	
	\begin{lem}\label{geo5}
		Let \( d_r \) be the number of words \( V \) of length \( \le r \) over the alphabet \( S \) in the group $\mathbb{A}_\mathcal{P}$ that are conjugate to some word \( W \) over the alphabet \( S' \). Then
		\[
		d_r < r \cdot |B_\mathbb{F}\left( \tfrac{1}{\overline{\beta}} r \right)| \cdot |B_{\mathbb{A}_\mathcal{P}}(3\gamma r)|.
		\]
	\end{lem}

	\section{On the growth of the free Burnside group of rank 3}\label{4}
	Let $G$ be a group generated by a finite set $S$. The growth function of $G$ with respect to $S$, denoted $\gamma_{G,S}(s)$, is the number of elements $g\in G$ that can be represented as a product of at most $s$ generators from $S$ or their inverses. By $\gamma_{B(m,n)}(s)$ and $\gamma_{F_m}(s)$ we denote the growth functions of the free Burnside group $B(m,n)$ and absolutely free group $F_m$ of rank $m$ with respect to its free generators.
	
	It is a well-established result by S.I. Adian that for odd $n\geq 665$ and $m>1$, the free Burnside group $B(m,n)$ has exponential growth \cite[Ch. VI, Theorem 2.15]{A}. Moreover, this growth is uniform, a stronger property that holds for all finite generating sets (see \cite{A09}, \cite{A109}, \cite{AA15}).
	
	For the main result of this paper, we require an explicit lower bound on the growth function of $B(3,n)$. This can be derived from a general estimate established in \cite{Bay}. The main theorem of that paper provides a lower bound for the growth function of $B(m,n)$ for any rank $m\geq 2$. Specifically, it is shown that
	\[
	\gamma_{B(m,n)}(s)> \frac{m}{m-1} (2m-1-\alpha)^{s}-1
	\]
	for any $\alpha\in (0,1)$ satisfying the condition $2m(2m-1)<\alpha(2m-1-\alpha)^8$. 
	
	Substituting $m=3$ and $\alpha=0.1$ into the inequality yields the specific bound we need. This leads to the following lemma.
	
	\begin{lem}[see Theorem 1, \cite{Bay}]\label{GrowthEstimate}
		For all odd $n\geq 665$ and natural $s$ 
		\[
		\gamma_{B(3,n)} (s)> \frac{3}{2} (4.9)^{s}-1.
		\]
	\end{lem}
	
	\section{Proof of Theorem \ref{T}}\label{5}
	
	\paragraph{Proof of Statement 2 of Theorem \ref{T}.} The identity $x^n = 1$ does not hold in the group $\mathbb{A}_\mathcal{P}$, since it fails in its subgroup $\mathbb{F}$. 
	
	\paragraph{Proof of Statement 1 of Theorem \ref{T}.} Let us prove that $x^n = 1$ is an $M$-probabilistic identity in the group $\mathbb{A}_\mathcal{P}$.

	We can estimate $\gamma_{\mathbb{A}_\mathcal{P}, S}(r)$ and $\gamma_{\mathbb{F}, S'}(r)$ from above using the growths $\gamma_{F_3}(r)$ and $\gamma_{F_2}(r)$ of the free groups $F_3$ and $F_2$ of rank 3 and 2 with respect to its free generators:
	\begin{equation}\label{3}
		\gamma_{\mathbb{A}_\mathcal{P}, S}(r) \leq \gamma_{F_3}(r) = 1+\frac{3(5^{r} - 1)}{2} < \frac{3}{2}\cdot 5^r,
	\end{equation}
	\begin{equation}\label{2}
		\gamma_{\mathbb{F}, S'}(r) \leq \gamma_{F_2}(r) = 2\cdot3^r-1 < 2\cdot 3^r.
	\end{equation}

	On the other hand, $\mathbb{A}_\mathcal{P}$ admits a homomorphism onto the free Burnside group $B(3, n)$, since $\mathbb{A}_\mathcal{P}$ is a 3-generator $n$-torsion group (see \eqref{e1}). Therefore, $\gamma_{\mathbb{A}_\mathcal{P}}(r)$ is no less than $\gamma_{B(3, n)}(r)$. Using the estimate for $\gamma_{B(3, n)}$ obtained in Section \ref{4}, we have:
	\[
	\frac{3}{2} \cdot (4.9)^{r}-1 <\gamma_{\mathbb{A}_\mathcal{P}, S}(r) < \frac{3}{2}\cdot 5^r. 
	\]
	
	From Lemma \ref{geo4} we immediately get the following estimate for the number $d_r$ of words $V$ of length $\le r$ in alphabet $S$ that are conjugate in the group $\mathbb{A}_\mathcal{P}$ to some word $W$ in alphabet $S'$:
	\[d_r\le r\cdot\gamma_{\mathbb{F}, S'}(\overline{\beta}^{-1}r)\cdot\gamma_{\mathbb{A}_\mathcal{P},S}(3\gamma r).\]
	
	Taking into account \eqref{3} and \eqref{2}, we obtain:
	\[
	d_r<r\cdot2\cdot 3^{\overline{\beta}^{-1}r}\cdot\frac{3}{2}\cdot 5^{3\gamma r}= 3r(3^{\overline{\beta}^{-1}}5^{3\gamma})^r.
	\]
	
	Using the least parameter principle (LPP), we choose the parameters $\beta$ and $\gamma$ small enough that the inequality \[3^{\overline{\beta}^{-1}}5^{3\gamma}<4\] holds (recall that $\overline{\beta}=1-\beta$).
	
	Then \[\frac{d_r}{\gamma_{\mathbb{A}_\mathcal{P}, S}(r)}<\frac{3\cdot r\cdot4^r}{\frac{3}{2} \cdot (4.9)^{r}-1}\xrightarrow[r\to\infty]{} 0.\]
	
	Thus, Statement 1 of Theorem \ref{T} is proven.

	\paragraph{Proof of Statement 3 of Theorem \ref{T}.}\label{6}
	
	Suppose that for some $ \mathcal{P} \neq \mathcal{P}'$, the groups $ \mathcal{A}_\mathcal{P} $ and $ \mathcal{A}_\mathcal{P'} $ are isomorphic. Then, by Lemma 5, the group $ \mathcal{A}_\mathcal{P} $ contains a subgroup that is isomorphic to $ \mathbb{F}(n, \mathcal{P}') $. In other words, $ \mathcal{A}_\mathcal{P} $ contains two elements $ u'_1, u'_2 $ that generate a subgroup isomorphic to $ \mathbb{F}(n, \mathcal{P}') $. Moreover, since for $ \mathcal{P}' \neq \mathcal{P}'' $, the groups $ \mathbb{F}(n, \mathcal{P}') $ and $\mathbb{F}(n, \mathcal{P}'') $ are not isomorphic, the corresponding pairs of generators $ (u'_1, u'_2) $ and $ (u''_1, u''_2) $ are distinct ($ \langle u'_1, u'_2 \rangle \simeq \mathbb{F}(n, \mathcal{P}') $, $ \langle u''_1, u''_2 \rangle \simeq \mathbb{F}(n, \mathcal{P}'') $). In the countable group $ \mathcal{A}_\mathcal{P} $ the set of distinct pairs of elements $ (u'_1, u'_2) $ is countable. Therefore, the set of groups isomorphic to the given group $ \mathcal{A}_\mathcal{P} $ is at most countable. Since the set of non-isomorphic groups $ \mathbb{F}(n, \mathcal{P}) $ for various subsets $ \mathcal{P} $ of the prime numbers set is uncountable (see \cite[Chapter VII]{A}, Proposition 2.17), it follows that the set of non-isomorphic groups $ \mathcal{A}_\mathcal{P} $ for different $ \mathcal{P} $ is also uncountable.
	
	\section*{Acknowledgements}
	The work of V.S. Atabekyan and V.H.Mikaelian is supported by the
	Higher Education and Science Committee of RA, research project no. 25RG-1A187, the work of A.A. Bayramyan is supported by the Higher Education and Science Committee of RA, research project no. 23AA-1A028.


\begin{thebibliography}{99}
		
		\bibitem{Gus} W. H. Gustafson. What is the Probability that Two Group Elements Commute?, The American Mathematical Monthly, 80(9), 1031–1034, 1973. https://doi.org/10.1080/00029890.1973.11993437
		
		\bibitem{GW} R.M. Guralnick and J.S. Wilson, The probability of generating a finite soluble group, Proc. London Math.
		Soc. (3) 81 (2000), no.2., 405--427.
		
		\bibitem{ABGK} Amir, G., Blachar, G., Gerasimova, M., Kozma, G. (2023). Probabilistic Laws on Infinite Groups (arXiv:2304.09144). arXiv. http://arxiv.org/abs/2304.09144
		
		\bibitem{YMV} Y. Antolın, A. Martino and E. Ventura.  Degree of commutativity of infinite groups. https://doi.org/10.48550/arXiv.1511.07269
		
		\bibitem{T} Matthew C.H.Tointon. Commuting probabilities of infinite groups. Journal of the London Mathematical Society,101(3):1280–1297, 2020.
		
		\bibitem{YMVp} Y. Antolın, A. Martino and E. Ventura. Degree of commutativity of infinite groups, Proc. Amer. Math. Soc. 145(2) (2017), 479--485.
		
		
		
		\bibitem{A82}S.I. Adian, Random walks on free periodic groups, Math. USSR-Izv., 21:3 (1983), 425–434
		
		\bibitem{O+}G. Goffer, B. E. Greenfeld, A. Yu. Olshanskii,  (2024). Asymptotic Burnside laws. 
		https://doi.org/10.48550/arXiv.2409.09630
		
		\bibitem{AB} V. S. Atabekyan, A. A. Bayramyan, Probabilistic Identities in n-Torsion Groups, Journal of Contemporary Mathematical Analysis (Armenian Academy of Sciences) \textbf{59}:6 (2024), 455-459.
		
		
		\bibitem{A76}S. I. Adian, Periodic products of groups, Proc. Steklov Inst. Math., \textbf{142} (1979), 1-19.
		
		\bibitem{A70i}S. I. Adian, Infinite irreducible systems of group identities, Math. USSR-Izv., \textbf{4}:4 (1970), 721–739
		
		\bibitem{A} S. I. Adian, \textit{The Burnside Problem and Identities in Groups}, Springer-Verlag, 1979.
		
		\bibitem{M07}
		V.H. Mikaelian, 
		On finitely generated soluble non-Hopfian groups, an application to a problem of Neumann, 
		International Journal of Algebra and Computation, 17 (2007), 5-6, 1107--1113.
		
		\bibitem{M10}
		V.H. Mikaelian, 
		On finitely generated soluble non-Hopfian groups, 
		Journal of Mathematical Sciences (Springer), 166 (2010), 6, 743--755. 
		
		\bibitem{M18}
		V.H. Mikaelian,
		Subvariety structures in certain product varieties of groups,
		Journal of Group Theory, 21 (2018), 5, 865--884.
		
		\bibitem{N67}
		H. Neumann, 
		\textit{Varieties of Groups},
		Ergebnisse der Mathematik und ihrer Grenzgebiete, Springer, Berlin, 1967.
		
		\bibitem{AA17}S. I. Adian, V. S. Atabekyan, On free groups in the infinitely based varieties of S. I. Adian, Izv. RAN. Ser. Mat., 81:5 (2017),  3–14. 
		
		\bibitem{AA17i} S. I. Adian, V. S. Atabekyan, Periodic product of groups, Journal of contemporary mathematical analysis, \textbf{52} (3) (2017), 111–117 .
		
		\bibitem{AA15} S.I. Adian, V.S. Atabekyan, Characteristic properties and uniform non-amenability of $n$-periodic products of groups , Izv. Math., 79:6 (2015), 1097--1110.
		
		\bibitem{AA19} S. I. Adian, V. S. Atabekyan. $n$-torsion groups. Journal of Contemporary Mathematical Analysis, \textbf{54} (6) (2019), 319–327. 
		
		\bibitem{Bo} N. S. Boatman, Partial-Burnside groups (Dissertation), 2012, 
		http://hdl.handle.net/1803/14921
		
		\bibitem{O91}A. Yu. Ol’shanskii, Geometry of defining relations in groups, Translated from the 1989 Russian original by Yu. A. Bakhturin. Mathematics and its Applications (Soviet Series), 70.
		Kluwer Academic Publishers Group, Dordrecht, 1991. 
		
		\bibitem{A09} V. S. Atabekyan, Uniform non-amenability of subgroups of free Burnside groups of odd period, Math. Notes, 85:4 (2009), 496--502.
		
		\bibitem{A109} V. S. Atabekyan, Monomorphisms of Free Burnside Groups, Math. Notes, 86:4 (2009), 457--462.
		
		
		
		\bibitem{Bay} A. Bayramyan, On Growth of Free Burnside Groups, Armen.J.Math., vol. 17, no. 8 (2025), 1–7 
		
		
		
		
		
	\end{thebibliography}
\end{document}